\documentclass[10pt]{amsart}
\usepackage{amsmath,amssymb,amscd,graphicx,amsthm,amsfonts,amstext,mathrsfs,upgreek,enumitem, stmaryrd, setspace, multicol, color, mathtools, bbm, dsfont, turnstile,pb-diagram, flafter, float, mathtools, verbatim}

\usepackage
[a4paper, margin=1.2in]{geometry}
\usepackage{hyperref}

\hypersetup{colorlinks=true}

\renewcommand{\L}{\mathcal{L}}

\newcommand{\E}{\mathcal{E}}
\newcommand{\M}{\mathcal{M}}

\newcommand{\PP}{\mathbb{P}}

\newcommand{\hull}[2]{\text{Hull}^{#1}(#2)}

\newcommand{\HHH}[1]{\mathrm{H}({#1})}

\newcommand{\w}{\omega}

\newcommand{\id}{\operatorname{id}}
\newcommand{\cp}[1]{\operatorname{crit}(#1)}

\newcommand{\otp}[1]{{{\rm{otp}}\left(#1\right)}}

\newcommand{\range}{\operatorname{ran}}

\newcommand{\Col}{\operatorname{Coll}}

\newcommand{\ult}[1]{\operatorname{Ult}(#1)}

\newcommand{\res}[2]{#1\!\upharpoonright\!{#2}}

\newcommand{\On}{{\mathrm{Ord}}}
\newcommand{\card}[1]{|#1|}

\DeclareMathOperator*{\bigdoublevee}{\bigvee\mkern-12mu\bigvee}

\newcommand{\ZFC}{{\sf ZFC}}

\newcommand{\OR}{{\sf OR}}

\newcommand{\ran}[1]{{{\rm{rng}}(#1)}}
\newcommand{\dom}[1]{{{\rm{dom}}(#1)}}

\newcommand{\LL}{L}

\newtheorem{theorem}{Theorem}[section]
\newtheorem{lemma}[theorem]{Lemma}

\newtheorem{proposition}[theorem]{Proposition}

\newtheorem{question}[theorem]{Question}

\newtheorem{claim}{Claim}
\newtheorem*{claim*}{Claim}
\newtheorem*{subclaim*}{Subclaim}

\theoremstyle{definition}
\newtheorem{definition}[theorem]{Definition}

\theoremstyle{remark}

\newenvironment{enumerate-(1s)}{\begin{enumerate}[label={\upshape (\arabic{*}*)}]}{\end{enumerate}}
\newenvironment{enumerate-(is)}{\begin{enumerate}[label={\upshape (\roman{*})*}]}{\end{enumerate}}
\newenvironment{enumerate-(iss)}{\begin{enumerate}[label={\upshape (\roman{*}**)}]}{\end{enumerate}}

\newenvironment{enumerate-(a)}{\begin{enumerate}[label={\upshape (\alph*)}]}{\end{enumerate}}

\newenvironment{enumerate-(5a)}{\begin{enumerate}[label={\upshape (6-\alph*)}]}{\end{enumerate}}

\newenvironment{enumerate-(a)-r}{\begin{enumerate}[label={\upshape (\alph*)}, resume]}{\end{enumerate}}

\newenvironment{enumerate-(A)}{\begin{enumerate}[label={\upshape (\Alph*)}]}{\end{enumerate}}

\newenvironment{enumerate-(A)-r}{\begin{enumerate}[label={\upshape (\Alph*)}, resume]}{\end{enumerate}}

\newenvironment{enumerate-(i)}{\begin{enumerate}[label={\upshape (\roman*)}]}{\end{enumerate}}

\newenvironment{enumerate-(i)-r}{\begin{enumerate}[label={\upshape (\roman*)},resume]}{\end{enumerate}}

\newenvironment{enumerate-(I)}{\begin{enumerate}[label={\upshape (\Roman*)},]}{\end{enumerate}}

\newenvironment{enumerate-(I)-r}{\begin{enumerate}[label={\upshape (\Roman*)},resume]}{\end{enumerate}}

\newenvironment{enumerate-(1)}{\begin{enumerate}[label={\upshape (\arabic*)}]}{\end{enumerate}}

\newenvironment{enumerate-(1)-r}{\begin{enumerate}[label={\upshape (\arabic*)},resume]}{\end{enumerate}}

\definecolor{teal2}{rgb}{0.036, 0.512, 0.512}

\usepackage{hyperref}
\hypersetup{colorlinks=true}
\hypersetup{citecolor=teal2,linkcolor=teal2}

\begin{document}
\nocite{*} 

\bibliographystyle{acm}



\author{Fabiana Castiblanco}\thanks{Both authors gratefully acknowledge support from the SFB 878 program ``\emph{Groups, Geometry \& Actions}'' financed by the DFG (Deutsche Forschungsgemeinschaft).}

\address{Fabiana Castiblanco, Institut f\"ur Matematische Logik und Grundlagenforschung, Universit\"at M\"unster, 
Einsteinstra{\ss}e 62, 48149 M\"unster, Germany}
\email{fabi.cast@wwu.de}

\author{Ralf Schindler}
\address{Ralf Schindler, Institut f\"ur Matematische Logik und Grundlagenforschung, Universit\"at M\"unster, 
Einsteinstra{\ss}e 62, 48149 M\"unster, Germany}
\email{rds@wwu.de}


\title{When is a given real generic over $L$?}

\begin{abstract} 
In this paper we isolate a new criterion for when a given real $x$ is generic over $L$ in terms of $x$'s capability of lifting elementary embeddings of initial segments of $L$.  \end{abstract} 

\maketitle

\setcounter{tocdepth}{1}

\section{Introduction}

The results established in \cite{cast} show that the property of closure under sharps for reals is preserved by certain tree forcing notions such as Sacks, Silver, Mathias, Miller and Laver by proving that enough ground-model embeddings $j:L[x]\to L[x]$, $x\in{}^\w\w$,  lift to the generic extension.  Here we turn our attention to a related problem:

\begin{question}Suppose that $0^\#$ exists.   Assume that $x\in{}^\w\w$ is such that every elementary embedding  $j:\LL\to \LL$ lifts to $j^*: \LL[x]\to \LL[x]$.  Is $x$ set generic over $\LL$?\end{question}

In this paper we will characterize the reals $x\in V$ which are generic by set forcing over $L$ by means of their capability to lift partial elementary embeddings of $L$ (see theorem \ref{converse}).  In order to prove our result, we present a version of Woodin's extender algebra for (partial) extenders which exist in $\LL$, and also we introduce the notion of ``weak Woodiness,'' which turns out not to be a large cardinal concept at all.

\section{Woodin's extender algebra, modified, and Bukowsk\'{y}'s Theorem}\label{woodinmodified}

\begin{definition}For a regular cardinal $\delta\geq\w_1$ and an ordinal $\mu\leq\delta$ let $\mathscr{L}_{\mu,\,\delta}$ be the least infinitary language 
which has constants $\check{\xi}$, all $\xi<\mu$, as well as $\dot{a}$, and which has
atomic formulae ``$\check{\xi}\in\dot{a}$", $\xi<\mu$, and is closed under 
negation and disjunction of length $< \delta$, i.e.,  
\begin{enumerate-(i)}\item if $\phi \in \mathscr{L}_{\mu,\,\delta}$, then $\neg \phi \in \mathscr{L}_{\mu,\,\delta}$, and 
\item if $\theta < \delta$ and $\phi_\alpha \in \mathscr{L}_{\mu,\,\delta}$ for all $\alpha<\theta$, then $\bigdoublevee_{\alpha<\theta}\phi_\alpha\ \in \mathscr{L}_{\mu,\,\delta}$.
%
\end{enumerate-(i)}
\end{definition}

Each $x\subseteq \mu$, $x$ not necessarily in $V$, defines a model for the logic $\mathscr{L}_{\mu,\,\delta}$.  Given $\varphi \in \mathscr{L}_{\mu,\,\delta}$ we may define the meaning of $x\models \varphi$ recursively:
\begin{enumerate-(1)}\item  $x\models ``\dot{\xi}\in\dot{x}$\,'' \, if and only if\:  $\xi\in x$,
\item $x\models\neg\varphi$\, if and only if\: $x\not\models\varphi$, and
\item $x\models \bigdoublevee\Gamma$\, if and only if\: $x\models \varphi$ for some $\varphi\in\Gamma$,  where  $\Gamma$ is an enumeration of formulas in $\mathscr{L}_{\mu,\,\delta}$ of length $<\delta$.
\end{enumerate-(1)}In this setting, notice that the statement $``x\models\varphi$'' is absolute between transitive models of $\ZFC$ containing $x$ and $\varphi$.

\begin{definition}Let $\PP\in V$ be a forcing notion and suppose that $g$ is $\PP$-generic over $V$.  \: For $\varphi \in \mathscr{L}_{\mu,\,\delta}$, let \[A^g_\varphi=\{x\in\wp(\mu)\cap V[g]: x\models \varphi\}\] Further, if $\Gamma\subseteq \mathscr{L}_{\mu,\,\delta}$ is a theory we set $A^g_\Gamma=\{x\in\wp(\mu)\cap V[g]: \text{ for all $\varphi\in \Gamma$, $x\models\varphi$}\}.$  \end{definition}
\begin{definition}Let $\PP\in V$ be a forcing notion.\:  We say that a theory $\Gamma$ in $\mathscr{L}_{\mu,\,\delta}$  is \emph{consistent} if and only if $A^g_\Gamma\neq\varnothing$ for some $\PP$-generic filter $g$.
If $\Gamma \cup \{ \phi \} \subset \mathscr{L}_{\mu,\,\delta}$, then we write
$\Gamma\vdash \varphi$ iff $\Gamma \cup \{ \neg \phi \}$ is inconsistent.
\end{definition}


\begin{lemma}For every theory $\Gamma\subseteq \mathscr{L}_{\mu,\,\delta}$ and every $\varphi\in \mathscr{L}_{\mu,\,\delta}$ the following are equivalent:
\begin{enumerate-(1)}\item $\Gamma\vdash \varphi$
\item $A^g_{\Gamma\cup\{\varphi\}}=A^g_\Gamma$\,  for any $g$ which is $\PP$-generic for a forcing notion $\PP$ which makes $\delta^{<\delta}$ countable.
\item $A^g_{\Gamma\cup\{\varphi\}}=A^g_\Gamma$\, for any $g$ which is $\Col(\w,\delta^{<\delta})$-generic over $V$.
\end{enumerate-(1)}
\end{lemma}
\begin{proof}Easy. See \cite[Lemma 2.2]{doe} or \cite[Lemma 1.2]{farah}.\end{proof}

\begin{definition}Let $\delta\geq \w_1$ be a regular cardinal such that $\delta=\delta^{<\delta}$ and let $\mu\leq\delta$.  Let $\mathcal{E}$ be a collection of elementary embeddings $j:M\to N$ with critical point $\kappa=\kappa_j$ such that $M$, $N$ are transitive, and $M\models \ZFC^-$, $\card{N}<\delta$, together with a map $\nu_\mathcal{E}:\mathcal{E}\to \textsf{OR}$ satisfying $\kappa_j+1\leq\nu_{\mathcal{E}}(j)\leq j(\kappa_j)$ for each $j$.     We associate to $\mathcal{E}$, in  addition  to  the  axioms  and  rules  for the infinitary  logic $\mathscr{L}_{\mu,\,\delta}$, a set of axioms $A_\mathcal{E}$ as follows:
\begin{align*} A_\mathcal{E}:\:\:\:\:& \text{Whenever $j\in\mathcal{E}$, $j:M\to N$ and $\vec{\phi}=(\phi_i:  i<\kappa_j)\in M$  then}\\ &\textstyle{\bigdoublevee \res{j(\vec{\phi})}{\nu_\E(j)}}\to \textstyle{\bigdoublevee}\vec{\phi}.\end{align*}
\end{definition}
\begin{definition}\label{lalg}Given $\mathcal{E}$ as above and $\phi$, $\psi \in\mathscr{L}_{\mu,\,\delta}$, we define \[\phi\sim_{\mathcal{E}}\psi\iff A_\mathcal{E}\vdash \psi \leftrightarrow\phi,\]
and we let $[\phi]=\{\psi: \psi \sim_{\mathcal{E}}\phi\}$.  We write  $\mathbb{P}_\mathcal{E} =\{[\phi]: \phi\in\mathscr{L}_{\mu,\,\delta}, \phi\text{ consistent with }A_\E \}$.  In $\PP_\mathcal{E}$ we stipulate \[[\phi]\leq[\psi] \iff  A_\mathcal{E}\vdash \phi\to\psi.\]
\end{definition}
\begin{definition}\label{defn_weak_woodin} We say that $\delta$ is \emph{weakly Woodin} as being witnessed by $\mathcal{E}$ iff  for all $A\subset \delta$ there is some $\kappa<\delta$ such that for all $\alpha<\delta$ there is some $j:M\to N$, $j\in\mathcal{E}$ with $\kappa_j=\kappa$, $A\cap \kappa\in M$ and $j(A\cap\kappa)\cap\alpha=A\cap\alpha$.  \end{definition}

 By taking hulls, it is easy to see that every regular cardinal $\delta\geq\w_1$ is weakly Woodin, so ``weak Woodinness''  is not a large cardinal concept.
 \begin{lemma}\label{conv3}Let $\delta\geq\w_1$ be regular.  Let $\theta>\delta$,
$p\in H_\delta$, and let $\mathcal{E}$ be the collection of all $j:M\to N$ such that $M$ 
and $N$ are transitive, $M\models \ZFC^-$, $\card{N}<\delta$ and $p\in H_{\kappa_j}^M$.  Then $\mathcal{E}$ witness that $\delta$ is weakly Woodin.\end{lemma}

\begin{proof}Let $A\subset\delta$, and let $X\prec H_{\delta^+}$ with $X\cap\delta\in\delta$, $p\in X$, $\card{X}<\delta$ and $A\in X$.   Let $i:M\cong X$ where $M$ is transitive.  Write $\kappa=\kappa(i)=X\cap\delta<\delta$.  Of course, $A\cap \kappa=i^{-1}(A)$. 

Let $\alpha<\delta$ and let $Y=\hull{H_{\delta^+}}{X\cup\{\kappa\}\cup(\alpha+1)}\prec H_{\delta^+}$.  Suppose  $k:N\cong Y$, where $N$ is transitive.  Setting $j=k^{-1}\circ i$, $j:M\to N$ has critical point $\kappa$, and $j(A\cap \alpha)=k^{-1}(A)$ and $k^{-1}(A)\cap\alpha=A\cap\alpha$ as $\res{k}{\alpha+1}=\id$.\:  So $j(A\cap\kappa)\cap\alpha=A\cap\alpha$.\end{proof}

 \begin{lemma}\label{deltacc}Suppose $\delta\geq\w_1$ is a regular cardinal with $\delta^{<\delta}=\delta$.   Let $\mathcal{E}$ witness that $\delta$ is weakly Woodin, and let $\nu:\mathcal{E}\to \OR$ be a map with $\kappa_j+1\leq\nu_{\mathcal{E}}(j)\leq j(\kappa_j)$ for each $j\in\mathcal{E}$.\:   Then $\mathbb{P}_\mathcal{E}$\: as being defined in \ref{lalg} with the theory $A_\mathcal{E}$,\: has the $\delta$-c.c. 
\end{lemma}
\begin{proof} Let $\vec{\phi}=\langle\phi_i:i<\delta\rangle$ be such that $\langle[\phi_i]:i<\delta\rangle$ is an antichain in $\mathbb{P}_{\mathcal{E}}$.     As $\delta^{<\delta}=\delta$ and every $\phi_i$ is a formula in $\mathcal{L}_{\mu,\delta}$, we may code $\vec{\phi}$ by a subset $A$ of $\delta$ (so $\vec{\phi}$ may be identified with $A$).\:  As $\delta$ is weakly Woodin,
we may pick $\kappa$ as in \ref{defn_weak_woodin}. We have $\kappa+1<\delta$, and we may pick $j:M\to N$ in $\mathcal{E}$ such that  if $\kappa=\kappa_j$, then
$$j(\res{\vec{\phi}}{\kappa})(\kappa)=\phi_{\kappa}$$
Since $\nu(j)\geq \kappa+1$, the axioms of $A_\mathcal{E}$ tell us  that:
$$A_\mathcal{E}\vdash\phi_\kappa\to\bigdoublevee \res{j(\res{\vec{\phi}}{\kappa})}{\nu(j)}\to\bigdoublevee \res{\vec{\phi}}{\kappa}$$
Theferore, $\{[\phi_i]:i\leq\kappa\}$ is not an antichain, which leads to a contradiction.
\end{proof}

For $x\subset \mu$ such that $x\models A_\E$, let $G_x=\{[\phi]\in\PP_\E: x\models\phi\}$. \:  It is easy to see that $G_x\subset \PP_\E$ is an ultrafilter.  Now,     given a $\PP_\E$-generic filter $G$ over $V$, notice that for $x:=\{\xi<\mu: \text{``$\dot{\xi}\in\dot{x}$''}\in G\}$ we have $G_x=G$.   In such case, we say that $x$ is $\PP_\E$-\emph{generic} over $V$. 

\begin{lemma}\label{conv2}Let $\mathcal{E}$ witness that $\delta$ is weakly Woodin,
where $\delta^{<\delta}=\delta$.  Let $x\subset\mu$, $x$ not necessarily in $V$ but in a transitive  outer model $V[x]$
of $\ZFC$, and assume that $x\models A_\mathcal{E}$.  Then $x$ is $\mathbb{P}_\mathcal{E}$-generic over $V$.
\end{lemma}

\begin{proof} We show that $G_x =\{[\varphi]\in \PP_\E:x\models \varphi\}$ is $\PP_\E$-generic over $V$. 
Let $A=\{[\phi_i]:i<\theta\}$ be a maximal antichain of $\mathbb{P}_\mathcal{E}$ in $V$.  By   \ref{deltacc} we have that $\PP_\E$ is $\delta$-c.c.  so  $\theta<\delta$ and hence $\bigdoublevee_{i<\theta} \phi_i$ is in $\mathscr{L}_{\mu,\delta}$ and in
fact $[\bigdoublevee_{i<\theta} \phi_i] \in \PP_\E$.   Since $A$ is maximal, we have that $A_\E\vdash\bigdoublevee_{i<\theta} \phi_i$. 

We claim that $x\models A_\E$ yields that \[x\models \bigdoublevee_{i<\theta} \phi_i.\]
Suppose that $x \models A_\E \cup \{ 
\neg \bigdoublevee_{i<\theta} \phi_i \}$. Then in $V[x]$ and hence in
$V[x]^{\Col(\w,\delta)}$ there is some $x'$ with $x' \models A_\E \cup \{ 
\neg \bigdoublevee_{i<\theta} \phi_i \}$, and by Shoenfield absoluteness there will be
an $x'$ in $V^{\Col(\w,\delta)}$ with $x' \models A_\E \cup \{ 
\neg \bigdoublevee_{i<\theta} \phi_i \}$. But this contradicts $A_\E\vdash\bigdoublevee_{i<\theta} \phi_i$. 

Thus $x\models \phi_i$ for some $i<\theta$, so $G_x\cap A\neq \varnothing$ and hence $G_x$ is $\PP_\E$-generic.  
   \end{proof}

\begin{lemma}\label{key_lemma} Let $\delta$, $p,\, \mathcal{E}$ be as in the statement of lemma \ref{conv3}.  Let $x\subset\mu$, $x$ not necessarily in $V$, be such that for every $j:M\to N$ in $\mathcal{E}$,  there is some elementary $\hat{j}:M[x]\to N[x]$ with $\hat{j}\supset j$ and $M[x]\models \ZFC^-$.  Then $x$ is $\mathbb{P}_\mathcal{E}$-generic over $V$.\end{lemma}

\begin{proof}By the previous lemmas, it suffices to show that $x\models A_\mathcal{E}$.  So let $j:M\to N$ in $\mathcal{E}$ and let $\vec{\phi}=(\phi_i:i<\kappa(j))\in M$.  Let us assume that $x\models \bigdoublevee \res{j(\vec{\phi})}{\nu_\E(j)}$.  Then $$N[x]\models ``x\models\bigdoublevee \res{{j}(\vec{\phi})}{\nu_\E({j})}"$$
hence by elementarity of $\hat{j}$, $M[x]\models ``x\models\bigdoublevee \res{\vec{\phi}}{\kappa}"$ and so $x\models\bigdoublevee\vec{\phi}$. \end{proof}

The arguments given so far allow us to reprove a theorem of Bukowsk\'{y}'s,
see \cite{bukovsky}, which we state as follows:

\begin{theorem}Let $X\subset \mu$, $X$ not necessarily in $V$ but in a transitive outer model $V[X]$ of $\ZFC$, and let $\delta$ be a regular uncountable cardinal.  The following are equivalent:

\begin{enumerate-(1)}\item There is some $\mathbb{P}$ such that $\mathbb{P}$ has the $\delta$-c.c. and $X$ is $\mathbb{P}$-generic over $V$. 
\item There is some $\theta>\mu$ and some club $C$ of $Y\prec \HHH{\theta}$ with $Y\cap\delta\in\delta$, $\card{Y}<\delta$ such that if $j:M\cong Y$, where $M$ is transitive, then there is some elementary $\hat{j}:M[\bar{X}]\to\HHH{\theta}[X]$ for some $\bar{X}$ and $\HHH{\theta}[X]\models \ZFC^-$.
\item If $f:\theta\to\On$, some $\theta$, $f\in V[x]$, then there is $g:\theta\to\wp(\On),\, g\in V$ such that $\card{g(\xi)}<\delta$ in $V$ and $f(\xi)\in g(\xi)$ for all $\xi<\theta$.\end{enumerate-(1)}\end{theorem}

\begin{proof}``(2)$\Rightarrow$(1)'':\: This is by the proof of Lemma \ref{key_lemma}. 

\noindent``(3)$\Rightarrow$(2)'':\:  Let $\theta>\mu$ be sufficiently big.\:  Notice that for all $X$,  $\HHH{\theta}\subset\HHH{\theta}^{V[X]}=\HHH{\theta}[X]$.\:  Doing a suitable book-keeping, let $\tilde{h}:\w\times(\HHH{\theta}^{V[G]})^{<\w}\to \HHH{\theta}^{V[G]}\in V[G]$ be a Skolem  function.  Thus, for all $A\subset\HHH{\theta}$,  $\tilde{h}"A\prec\HHH{\theta}^{V[G]}$.   

\noindent Let us define $h:\theta\to \On$ by using $\tilde{h}$ as follows:

\[h(x)=\begin{cases}\tilde{h}(x)&\text{ if  }\tilde{h}(x)\in \HHH{\theta}\\
\varnothing&\text{ otherwise}\end{cases}\]

\noindent By using (3), we may find some $g\in V$ such that for all $\xi<\theta$, $\card{g(\xi)}<\delta$ and $h(\xi)\in g(\xi)$.\:  In particular, for every $A\subset\HHH{\theta}$ of size $<\delta$ in $V$, $g"A\cap\delta\in\delta$ and $h"A\subset g" A$.

\noindent Notice that if $\tilde{h}(x)\in V$,  $x\in A$, then $\tilde{h}(x)=h(x)\in g(x)\subset A$, so  $\tilde{h}"A\cap V\subset A$.  On the other side, as $\tilde{h}$ is a Skolem function we have that $A\subset \tilde{h}"A\cap V$ for $A\in V$.  Thus if $A\in V$ then $\tilde{h}"A\cap V= A$.   This shows that for every $A\in V$,  $A\subset \HHH{\theta}$, $A=g"A$ holds.   
Therefore, as $\tilde{h}"A\prec \HHH{\theta}^{V[G]}$ and $\tilde{h}"A\cap V=A$, we have that $A\prec \HHH{\theta}^V$.

\noindent According to this, let us define $C$ as the  collection of all $g$-closed $Y\prec \HHH{\theta}$ with $Y\cap\delta\in\delta$ and $\card{Y}<\delta$.\:  By construction, we have that the collection $C$ satisfy (2).

\noindent ``(1)$\Rightarrow$(3)'':\: This is a standard argument.  \end{proof}

\begin{question}Suppose that $\mathcal{E}\subset V_\delta$ is a collection of extenders.  Define $A_\mathcal{E}$ as in section \ref{woodinmodified}, that is for each extender $E\in \mathcal{E}$ and for every sequence of formulae $\vec{\varphi}=(\varphi_i: i<\cp{E})$ we associate the axiom
$$\bigdoublevee \res{i_E(\vec{\varphi})}{\nu(E)}\to \bigdoublevee \vec{\varphi}$$
where $i_E: M\to \text{Ult}(M, E)$ and $\nu(E)$ is the strength of the extender.  Assume that $\mathbb{P}_\mathcal{E}$ has the $\delta$-c.c.  Is $\delta$ Woodin?\end{question}

\section{A criterion for set-genericity over $L$}
  
\begin{proposition}[Kunen]\label{kunenanc}Let $j:\LL\to \LL$ be an elementary embedding with critical point $\kappa$.  Then, for all $\alpha<\kappa^{+\LL}$, $\res{j}{\LL_\alpha}\in \LL$.
\end{proposition}

\begin{proof} This is by the ``ancient Kunen argument.'' \:   Let $\alpha<\kappa^{+\LL}$ and pick $f:\kappa\to \LL_\alpha$ onto, $f\in \LL$.   Note that for every $x\in \LL_\alpha$, $j(x)=y$ if and only if 
\[\exists\xi<\kappa (x=f(\xi) \wedge y=j(f)(\xi))\] 
Since $j(f)\in \LL$, for every $x\in \LL_\alpha$ we can compute $j(x)$ in $\LL$ from $f$ and $j(f)$.  Therefore, $\res{j}{\LL_\alpha}\in \LL$ as required. 
\end{proof}

\begin{theorem}\label{converse} The following are equivalent for a given real $x\in V$:

\begin{enumerate-(1)}\item $x$ is set-generic over $\LL$
\item There is some $p\in \LL$ such that for all elementary $j:\LL_\alpha\to \LL_\beta$ with critical point $\kappa$, where $j\in \LL$, $\LL_\alpha\models \ZFC^-$ and $p\in \LL_\kappa (\subsetneq \LL_\alpha)$, there is some $\hat{j}:\LL_\alpha[x]\to \LL_\beta[x]$ with $\hat{j}\supset j$ and $\LL_\alpha[x]\models \ZFC^-$.
\end{enumerate-(1)}
\end{theorem}

\begin{proof}``$\Leftarrow$'':\:  Given $p$, let $\delta\geq\w_1$ be regular with $p\in \LL_\delta$.  Let $\mathcal{E}$ be defined in $\LL$ as in the statement of lemma \ref{conv3}.  Then if $\mathbb{P}_\mathcal{E}$ is defined as in definition \ref{lalg} inside $L$, $x$ is $\mathbb{P}_{\mathcal{E}}$-generic over $\LL$.

\item ``$\Rightarrow$'':\: Let $\mathbb{P}\in \LL$ be such that $x$ is $\mathbb{P}$-generic over $\LL$.  Write $\mu=\text{Card}^\LL(\mathbb{P})$ and let $p= \LL_{\mu^+}$, where $\mu^+=\mu^{+\LL}$.  Let $j:\LL_\alpha\to \LL_\beta$ with $\cp{j}=\kappa$, $\mu^+<\kappa$, $\LL_\alpha\models \ZFC^-$.

\noindent Without loss of generality, let us assume that $\mathbb{P}\in p\in \LL_\kappa$, so that $x$ is $\mathbb{P}$-generic over $\LL_\alpha$.\:  This gives $\LL_\alpha[x]\models \ZFC^-$.  The real $x$ is also $\mathbb{P}$-generic over $\LL_\beta$, and by writing $X=\ran{j}$, $\mathbb{P}\in p\in X$.\:    In order to see that there is $\hat{j}:\LL_\alpha[x]\to \LL_\beta[x]$ with $\hat{j}\supset j$ it suffices to verify that $X[x]\cap\beta=X\cap\beta$\footnote{We pretend $x$ is the $\mathbb{P}$-generic filter.}.  Let $\tau\in X\cap \LL^\mathbb{P}$ be a name for an ordinal.  Note that $B=\{\gamma:\exists p\in\mathbb{P}\: (p\Vdash^\mathbb{P}_{\LL_\beta}\tau=\hat{\gamma})\}\in X$ and  $\otp{B}<\mu^+$.\: Thus, the order isomorphism $\pi:\otp{B}\cong B$ is in $X$, and since $\mu^+<\kappa\subset X$, we have $\otp{B}\cup\{\otp{B}\}\subseteq X$.  Hence $B\subseteq X$ and, in particular, $\tau^x\in X$.  \end{proof}

The paramater $p$ as in the statement (2) above is necessary.\:  Suppose that $j:\LL_\alpha\to \LL_\beta$ is an elementary embedding with critical point $\kappa$, and let $x$ be $\Col(\w,\kappa)$-generic over $\LL$.  In this case, such an embedding cannot be lifted to a $\hat{j}:\LL_\alpha[x]\to \LL_\beta[x]$ because $\cp{j}^{\LL_\alpha[x]}$ is countable.\: 

By Theorem \ref{converse}, if $x \in {\mathbb R} \cap V$ is such that (2)
of the statement of Theorem \ref{converse} holds true and $0^\#$ exists,
then $x^\#$ exists.

\bibliography{paper2}
  
\end{document}